\documentclass[a4paper,12pt,reqno]{amsart}
\usepackage{amsmath,amsfonts,amssymb}
\usepackage[style=numeric]{biblatex}
\usepackage{verbatim}
\usepackage{dsfont}
\usepackage{url}
\usepackage[colorlinks]{hyperref}
\usepackage{mathtools}
\usepackage[latin1]{inputenc}
\usepackage{enumitem}
\usepackage[english]{babel}
\usepackage{tikz}
\usepackage[margin=3.5cm]{geometry}
\usepackage{bbm}
\usepackage{multirow}
\usepackage{mathrsfs}
\usepackage{caption}
\usepackage{float}
\restylefloat{table}
\usepackage{graphicx}

\newcommand{\Z}{\mathbb Z}

\newcommand{\Tr}{\mathrm{Tr}}

\newcommand{\fqn}{\mathbb{F}_{q^n}}
\newcommand{\F}{\mathbb{F}}

\newcommand{\fq}{\mathbb{F}_q}

\newtheorem{theorem}{Theorem}[section]

\newtheorem{proposition}[theorem]{Proposition}
\newtheorem{definition}[theorem]{Definition}

\newtheorem{lemma}[theorem]{Lemma}
\newtheorem{corollary}[theorem]{Corollary}
\newtheorem{example}[theorem]{Example}

\DeclareMathOperator{\rank}{rank}

\author{Jos\'e Alves Oliveira}
\author{Daniela  Oliveira }
\author{F. E. Brochero Mart\'{i}nez}

\title{The number of rational points of a class of superelliptic curves}
\begin{document} 

\maketitle 

\begin{abstract}
In this paper, we study the number of $\fqn$-rational points on the affine curve $\mathcal{X}_{d,a,b}$ given by the equation
$$
y^d=ax\Tr(x)+b,
$$
where $\Tr$ denote the trace function from $\fqn$ to $\fq$ and $d$ is a positive integer. In particular, we present bounds for the number of $\fq$-rational points on $\mathcal{X}_{d,a,b}$ and, for the cases where $d$ satisfies a natural condition, explicit formulas for the number of rational points are obtained. Particularly, a complete characterization is given for the case $d=2$. As a consequence of our results, we compute the number of elements $\alpha$ in $\fqn$ such that $\alpha$ and $\Tr(\alpha)$ are quadratic residues in $\fqn$. 
\end{abstract}

\section{Introduction}
Let $\fqn$ be a finite field with $q^n$ elements, where $q=p^s$ and $p$ is an odd prime. Throughout the paper,   $\Tr(x)$ denotes the trace function from $\fqn$ into $\fq$. The study of the number and existence of special elements in finite fields dates back to the 1980s. The famous \textit{Primitive
Normal Basis Theorem} was firstly proved by Lenstra
and Schoof~\cite{lenstra1987primitive} in 1987.

More recently, sophisticated techniques have been created and employed in different problems regarding elements that satisfy special conditions (for example, see \cite{MR4252114,cohen2021primitive, cohen2019lehmer, gupta2018primitive}). The condition of having a fixed trace often appears in such articles because of its practical applications.

A problem that naturally arises is to find the number of elements $\alpha\in\fqn$ such that $\alpha$ and $\Tr(\alpha)$ are both quadratic residues in $\fqn$. This problem is easily solved in the case where $n$ is even, since any element of $\fq$ is a quadratic residue in $\fqn$, so that, in particular, $\Tr(\alpha)$ does so. The problem gets more interesting for the case where $n$ is odd. Heuristically, $q^n/4$ elements in $\fqn$ must satisfy these two conditions, but there is no result in the literature addressing this problem. In both cases, such number of special elements is closely related to the number of $\fqn$-rational points on the affine curve $y^2=x \Tr(x)$ (see the proof of Theorem~\ref{teores}). In this paper, we study a broader class of affine curves given by 
\[\mathcal{X}_{d,a,b}:y^d = ax\Tr(x) + b,\]
where $a,b \in \fqn$ and $d$ is a positive integer.
The curve $\mathcal{X}_{d,a,b}$ belongs to a wide family of curves called \textit{superelliptic} curves, whose points are given by the solutions of the equation
\begin{equation}\label{superellipiceq}
    y^d=f(x).
\end{equation}

There are few results in the literature regarding general superelliptic curves. For example, in \cite{cheong2005}  the authors give the distribution of points on smooth superelliptic curves over a  finite field, when their degree goes to infinity. In  \cite{kurlberg2009}  the authors describe the fluctuation in the number of points on a hyperelliptic curve, which is the especial class of superelliptic curves with $d=2$.

In general, it is hard to compute the number of rational points on superelliptic curves, but this can be done in some cases, namely, for those that arise by choosing a suitable polynomial $f$ in \eqref{superellipiceq}. For example, if $f(x)=-x^m+b$, the curve is the well-known Fermat curve, for which explicit formulas and bounds for the number of points are known (\cite{oliveira2021diagonal,wolfmann1992number}). The curve defined by Equation~\eqref{superellipiceq} over $\fqn$ with $f(x)=x^q-x+b$ is known as Artin-Schreier curve. These special superellipic curves have also been well studied \cite{coulter1998explicit,coulter2002number,wolfmann1989number}.

While these particular cases are well studied, a study of the number of rational points on $\mathcal{X}_{d,a,b}$ has not been provided. Our goal in this paper is to study the number of rational points on this curve, providing bounds for the number of rational points and explicit formulas for cases where $\mathcal{X}_{d,a,b}$ satisfy certain conditions.

In order to do that, we use the fact that the map $x\mapsto x\Tr(x)$ is a quadratic form over $\fqn$. Along the paper, we employ some classical results on quadratic forms over finite fields to provide an expression for the number of rational points on $\mathcal{X}_{d,a,b}$ in terms of Gauss sums (Proposition~\ref{termsofGauss}). Using this expression, we employ results on Gauss sums in order to obtain bounds for the number of rational points and, for suitable conditions on $d$, provide explicit formulas for this number. As a consequence of our results, we compute the number of elements $\alpha$ in $\fqn$ such that both $\alpha$ and $\Tr(\alpha)$ are quadratic residues in $\fqn$ (Theorem~\ref{teores}).

The paper is organized as follows. In Section~\ref{sec2}, we present some remarks, comments and statements of our main results. Section~\ref{sec3} provides preliminary results on quadratic forms and Gauss sums that are used along the paper. In   Section~\ref{sec4}, we provide an expression for the number of rational points on the curve  $\mathcal{X}_{d,a,b}$ in terms of Gauss sums. Bounds and explicit formulas for the number of rational points are presented in Section~\ref{sec5}. In Section~\ref{sec7}, we focus on the case $\Tr(b/a)\neq 0$.

\section{Main results}\label{sec2}
In this section, we present  our main results, along with a few comments. We observe that the curves $\mathcal X_{d,a,b}(\fqn)$ and $\mathcal X_{\gcd(d,q^n-1),a,b}(\fqn)$ have the same number of $\fqn$-rational points. Therefore, we may assume without loss of generality that $d$ divides $q^n-1$. 

In \cite{stohr1986weierstrass}, the authors give an improvement of the Hasse-Weil bound for curves with high genus. The curve $\mathcal{X}_{d,a,b}$ does not satisfy the conditions imposed in \cite{stohr1986weierstrass}, but as a high genus curve, it is expected to it to be a curve whose number of points is far from Hasse-Weil's bound. Indeed, it turns out that Hasse-Weil can be significantly improved for  $\mathcal{X}_{d,a,b}$. In Theorem~\ref{theoremdcota}, we provide sharp bounds for the number of rational points on such curves. In order to present this result, we introduce some notation. For a divisor $d$ of $q^n-1$, $\chi_d$ denotes a multiplicative character of $\fqn^*$ with order $d$. As usual, $\mathcal{X}_{d,a,b}(\fqn)$ denotes the set of $\fqn$-rational points on $\mathcal{X}_{d,a,b}$.

\begin{definition}\label{vD}
Let $v=\gcd\big(d,\tfrac{q^n-1}{q-1}\big)$, $D=\frac{d}{v}$ , $B=\Tr\big(\frac{b}{a}\big)$ and 
$$\tau = \begin{cases} 
	1, & \text{ if } p\equiv 1 \pmod{4};\\
	i,  & \text{ if } p\equiv 3 \pmod{4}.\\
\end{cases}$$
\end{definition}
Now we are able to present bounds for the number of $\fqn$-rational points on the curve $\mathcal X_{d,a,b}$
\begin{theorem}\label{theoremdcota} The number of rational points on the curve $\mathcal X_{d,a,b}$ satisfies the following relations:
\begin{equation}|\mathcal{X}_{d,a,b}(\fqn)|-q^n+(-1)^sq^{n-1}\sum_{\ell=1}^{d-1} \chi_d^\ell(b)\eqqcolon N_{d,a,b,n},\label{N_dabn}
\end{equation}
where $$|N_{d,a,b,n}|
\le\begin{cases}
		(\tfrac{d}{D}-1)(q-1)q^{\frac{n}{2}-1}, &\text{ if }D\text{ is odd and } B=0;\\
		(q-1)(\tfrac{d}{D}q^{\frac{n-1}{2}}+(\tfrac{d}{D}-1)q^{\frac{n}{2}-1}), &\text{ if }D\text{ is even and } B=0.\\
(d-1)(q^{\frac{n}{2}}+q^{\frac{n-1}{2}}) & \text{ if } B \neq 0.
	\end{cases}$$
	
\end{theorem}
We recall that the Hasse-Weil bound applied for $\mathcal X_{d,a,b}$ implies that
$$||\mathcal X_{d,a,b}(\fqn)|-q^n|\le 2g q^{\frac{n}{2}},$$
where the genus $g$ can be computed via   Riemann-Hurwitz's formula~\cite{galbraith2002arithmetic} and equals, at a minimum, the number
$$\tfrac{1}2(d-1)(q^{n-1}-1)-d+1.$$
Note that Theorem~\ref{theoremdcota} implies that
$$||\mathcal X_{d,a,b}(\fqn)|-q^n|\le (d-1)q^{n-1} + (q-1)(dq^{\frac{n-1}{2}}+(d-1)q^{\frac{n}{2}-1}),$$
which represents a significant improvement of Hasse-Weil's bound.

From now on, our main goal is to provide explicit formulas for the value $N = N_{d,a,b,n}$ defined in Theorem~\ref{theoremdcota}. In particular, the results obtained here include curves whose number of points attain the bounds provided in Theorem~\ref{theoremdcota}. Our starting point is the case $d=2$, for which we present a simple expression for 
the number of affine rational points, that is given in the next theorem.    

\begin{theorem}\label{theoremdcase2} 
The number of rational points on the curve $\mathcal X_{2,a,b}(\fqn)$ is
\begin{enumerate}\item If $B=0$  
$$\mid\mathcal{X}_{2,a,b}(\fqn)\mid=q^{n} + q^{n-1}\chi_2(b) + N_1\tau^{ns}\chi_2(-a)(q-1),$$
where 	$$N_1=\begin{cases}
	q^{\frac{n-2}2},&\text{ if }n\text{ is even;}\\
	\tau^{s}q^{\frac{n-1}2},&\text{ if }n\text{ is odd.}\\
	\end{cases}$$
	
\item	If $B\neq0$  
$$\mid\mathcal{X}_{2,a,b}(\fqn)\mid=q^{n} + q^{n-1}\chi_2(b) - N_2\tau^{ns},$$
where 
	$$N_2=\begin{cases}
	q^{\frac{n-2}2}[\chi_2(-a)\tau^{2s}q+\chi_2(-aB)],&\text{ if }n\text{ is even;}\\
		q^{\frac{n-1}2}\tau^s(\chi_2(-a)+\chi_2(-aB)),&\text{ if }n\text{ is odd.}\\
	\end{cases}$$
\end{enumerate}
\end{theorem}
This result allows us to compute the number of rational points in specific curves, as in the following examples.

\begin{example} 
For the curve
\[\mathcal X_{2,-1,0} : y^2 = -x\Tr(x),\]
it follows that $B = 0$. Therefore, Theorem~\ref{theoremdcase2} provides the number of rational points for a pair $(q,n)$:
$$\begin{tabular}{|c|c|c|c|c|c|}
\hline 
q & n=2 & n=3 &n=4  & n=5 & n=6  \\ \hline
3 &7  & 33 & 87 &225 & 711\\ \hline 
5 & 29 &145  &645  &3225 &15725 \\ \hline 
7 & 43& 385 & 2443& 16513 &117355\\\hline 
9 & 89 &801   & 6633 & 59697   &532089\\\hline 
11 & 111 & 1441 & 14751 & 159841 &1770351\\\hline 
25 & 649 & 16225 & 391225 &9780625  &254281250\\\hline 
\end{tabular}$$
\end{example}
\begin{example}
For the curve
\[\mathcal X_{2,-1,0} : y^2 = -x\Tr(x)+1,\]
it follows that $B = \Tr(-1)=-n$. Therefore, Theorem~\ref{theoremdcase2} provides the number of rational points for a pair $(q,n)$:

$$\begin{tabular}{|c|c|c|c|c|c|}
\hline 
q & n=2 & n=3 &n=4  & n=5 & n=6  \\ \hline
 3& 10 &  42 & 96 &385 & 954\\ \hline
 5 &34 &150 & 720 &  3850& 18600\\ \hline 
7 &64 & 378 &2688 & 19306&134162\\\hline 
9 &80  &  882  &7200 &65448 & 591138\\\hline 
11 & 122 & 1452 & 15840 &175692  &1931402\\\hline 
25 & 625 &16200  & 405600 & 10171250 &253890000\\\hline 
\end{tabular}$$
\end{example}

One can check some values obtained in these two examples by using a computer program such as SageMath. For instance, in the case $q=25$ and  $n=6$ using  a program in SageMath,  spent 37h 34min.
Nevertheless, the run time of the algorithm grows as $q^n$ increases, making the computation unfeasible even for small values of $q^n$ (such as $q^n>e^{20}$).

Back to the problem presented at the introduction, Theorem~\ref{theoremdcase2} is a key tool to compute the number of a special type of elements of $\fqn$, as we see in the proof of the following result.
 
\begin{theorem}\label{teores}
    Let $n$ be an odd positive integer. Then the number of elements $\alpha\in\fqn$ such that $\alpha$ and $\Tr(\alpha)$ are both quadratic residues in $\fqn$ is
    $$\frac{q^n+q^{n-1}+\tau^{s(n-1)}(q-1)q^{\frac{n-1}{2}}+2}{4}.$$
\end{theorem}

 In order to compute $N$, we present the following important definition, that is a generalization of a constant used by Wolfmann in \cite{wolfmann1992number} in the study of diagonal equations.
\begin{definition}For $m$ a divisor of $q^n-1$, $a\in\fqn^*$ and $\epsilon\in\{1,-1\}$, we set
 $$\theta_m(a,\epsilon)=\begin{cases}
		m-1,&\text{ if }\chi_m(a)=\epsilon;\\
		-1,&\text{ otherwise.}
	\end{cases}$$
\end{definition}

The following definition, that was introduced in the study of diagonal equations~\cite{oliveira2021diagonal}, will be useful in our results.

\begin{definition}
	Let $r$ be a positive integer. An integer $d>2$ is $(p,r)$-admissible if $d\mid (p^r+1)$ and there exists no $r'<r$ such that $d\mid (p^{r'}+1)$.
\end{definition}
If $d>2$ is $(p,r)$-admissible, then $2r$ is the multiplicative order of $p$ module $d$. Since $d$ divides $q^n-1$ and $q= p^s$, the condition of $(p,r)$-admissibility on $d$ implies that $2r$ divides $ns$ and, in particular, that $ns$ is an even number. 

\begin{definition}If $d$ is a divisor of $q^n-1$ and is $(p,r)$-admissible we define 
$$\varepsilon=(-1)^{\frac{ns}{2r}}\quad \text{ and } \quad    u=\frac{p^r+1}{d} .$$
\end{definition}
We present now one of our main results, which provides a formula for the number of rational points on $\mathcal X_{d,a,b}$ in the case where $B=0$ and some suitable conditions are required.
\begin{theorem}\label{casebeuquals0}
Let $d>2$ be an integer,
 $a,b \in \fqn$ such that $B=0$ and $N=N_{d,a,b,n}$ defined as in Equation \eqref{N_dabn}. The following holds:
	\begin{enumerate}
		\item If $v$ is $(p,r)$-admissible and $D$ is odd, then
		$$N=\begin{cases}
			0,&\text{ if }v=1;\\
		\tau^{sn}\chi_2(a) (q-1)q^{\frac{n-2}{2}},&\text{ if }v=2;\\
			\varepsilon\theta_{v}(-a,\varepsilon^{u})(q-1)q^{\frac{n-2}2},&\text{ if }v>2.\\
		\end{cases}$$
		\item If $2v$ is $(p,r)$-admissible and $D$ is even.
		\begin{enumerate}
		    \item If $v=1$ then
		    $$N=(-1)^s\tau^{(n+1)s}\chi_2(a)(q-1)q^{\frac{n-1}2}.$$
		    \item If $v=2$ then
		    $$N=\big[(-1)^{s}\tau^s\sqrt{q}\varepsilon^{u\frac{D}2+1}(\chi_4(-a)+\overline{\chi}_4(a))+\tau^{ns}\chi_2(a)\big](q-1)q^{\frac{n-2}2}.$$

\item If $v>2$ then
\[N = \varepsilon\Big[(-1)^s\tau^s\sqrt{q}\chi_{2v}(-a)\varepsilon^{\frac{uD}2}[1+\theta_{v}(-a,1)]+ \theta_{v}(-a,1)\Big](q-1)q^{\frac{n-2}2}.\]
		\end{enumerate}
	\end{enumerate}
\end{theorem}

By employing this result, we can obtain a simple expression for the number of rational points in the case where $a,b$ and $d$ satisfy some restrictions.
\begin{corollary}
    If $d>2$ is a $(p,r)$-admissible divisor of $\tfrac{q^n-1}{q-1}$ and $ns/2r$ is even, then
    $$N_{d,-1,0,n}=(d-1)(q-1)q^{\frac{n-2}{2}}.$$
\end{corollary}

Results similar to Theorem~\ref{casebeuquals0} can be obtained for the case $B\neq 0$, as we will see in Section~\ref{sec7}.

\section{Preliminary results}\label{sec3}
In this section, we provide some definitions and results that will be useful in the paper. Along the text, $\psi$ and $\tilde \psi$ denote the canonicals additive characters of $\fqn$ and $\fq$, respectively, i.e.,  $$\psi(x)=\exp\left(\tfrac {2\pi i\, \Tr_{\F_{q^n}/\F_p}(x)}p \right)\quad\text{and}\quad  \tilde\psi(x)=\exp\left(\tfrac {2\pi i\, \Tr_{\F_{q}/\F_p}(x)}p \right).$$
We use $\chi_{q^n-1}$ to denote a fixed primitive multiplicative character of $\fqn^* $ and, for $m$ a divisor of $q^n-1$, $\chi_m$ denotes the multiplicative character of order $m$ defined by $\chi_m = \chi_{q^n-1}^{(q^n-1)/m}$. The restriction of $\chi_m$ to $\fq^*$ is a multiplicative character of $\fq^*$ of order $M= \frac{m}{\gcd(m,(q^n-1)/(q-1))}$ and it will be denoted by $\eta_M$. 
\begin{definition}  For multiplicative characters $\chi$ of $\fqn^*$ and $\eta$ of $\fq^*$, the Gauss sum of $\chi$ and $\eta$ are the sums
\[G_n(\chi) = \sum_{x \in \fqn^*} \chi(x)\psi(x) \quad \text{ and } \quad  G_1(\eta) = \sum_{x \in \fq^*} \eta(x)\tilde\psi(x),\]
respectively.
\end{definition}

We present now some techinical results that are used to  compute the number of rational points on $\mathcal X_{d,a,b}$. Most of them are well-known and can be easily found in the literature.
\begin{lemma}[{{\cite[Theorem $5.4$]{Lidl}}}]\label{item433}
    Let $\chi$ be a multiplicative character of $\fqn$. Then
    $$\sum_{c\in\fqn}\chi(c)=\begin{cases}
        0, &\text{ if }\chi\text{ is nontrivial;}\\
        q^n, &\text{ if }\chi\text{ is trivial.}\\
    \end{cases}$$
\end{lemma}

\begin{lemma}\cite[Theorem 5.4]{Lidl}\label{sumlemm}
	For $u\in\fqn$, we have that
	$$\frac{1}{q^n}\sum_{c\in\fqn} \psi(uc)=\begin{cases}
	0, & \text{ if } u \neq 0;\\
	1,& \text{ if } u = 0.
	\end{cases}$$
\end{lemma}

\begin{lemma}[{{\cite[Equation $5.4$, p. 189]{Lidl}}}]\label{dpower}
Let $d$ be a divisor of $q^n-1$. If $c \in\fqn$, then
\[\sum_{j=0}^{d-1} \chi_d^j(c) = \begin{cases}
    1, & \text{ if } c=0;\\
	d, & \text{ if } c \text{ is a } d \text{-power in } \fqn^*;\\
	0,& \text{ otherwise}.
\end{cases}\]
\end{lemma}

\begin{lemma}\cite[Theorems 5.11 and 5.12]{Lidl}\label{gausslemm} 
	Let $\chi_0$ denote the trivial multiplicative character of $\fqn^*$. If $\chi\neq\chi_0$ is a multiplicative character of $\fqn^*$, then
	\begin{enumerate}
		\item $G_n(\chi_0)=-1;$
		\item $|G_n(\chi)|=q^{n/2};$
		\item $G_n(\chi)G_n(\overline{\chi})=\chi(-1)q^n$.
	\end{enumerate}
\end{lemma}
A Gauss sum is said to be  \textit{pure} if some positive integer power of it is real. The following result present necessary and sufficient conditions for a Gauss sum to be pure.
\begin{theorem}\cite[Theorem 1]{evans1981pure}\label{item405} 
	Given a divisor $d>2$ of $q^n-1$ and a multiplicative character $\chi_d$ of $\mathbb{F}_{q^n}^*$ with order $d$, the following are equivalents:
	\begin{itemize}
		\item there exists $r$ such that $d\mid (p^r+1)$;
		\item $G_n(\chi_d^j)$ is pure for all $j\in\Z$;
		\item there exists a positive integer $r$ such that $d\mid (p^r+1)$,  $2r\mid ns$ and $$G_n(\chi_d^j)=-(-1)^{ns(uj+1)/2r} q^{n/2}$$
		for all $j\not\equiv 0\pmod{d}$, where $u=\tfrac{p^r+1}{d}$.
	\end{itemize}
\end{theorem}
 
\begin{theorem}\cite[Theorem 5.15]{Lidl}\label{gaussdois}
	Let $\chi_2$ be the quadratic character of $\fqn^*$. Then
	$$G_n(\chi_2)=-(-1)^{sn}\tau^{sn}q^{n/2}.$$
\end{theorem}

\begin{corollary}\label{corolchartwo}
	Let $\chi_2$ be the quadratic character of $\fqn^*$. If $ns$ is even, then
	$$G_n(\chi_2)=-(-1)^{ns(u+1)/2} q^{n/2},$$
	where $u=\tfrac{p+1}{2}.$
\end{corollary}

%

\subsection{Quadratic forms}

 In order to determine the number of rational points on $\mathcal X_{d,a,b}$ we associate this curve to the quadratic form  $\Tr(x\Tr(x))$. From this quadratic form, we provide its associate matrix and the dimension of its radical. In order to do that, we recall the following definitions.
 \begin{definition}
For a quadratic form  $\Phi:\fqn \to \fq $, the symmetric bilinear form $\varphi:\fqn\times \fqn\to \fq$ associated to  $\Phi$ is  $$\varphi(\alpha,\beta) = \frac{1}2\left(\Phi(\alpha+\beta) -\Phi(\alpha)-\Phi(\beta)\right).$$ 
The radical of the symmetric bilinear form $\Phi: \fqn \to \fq$  is the set
\[\text{rad}(\Phi) = \{\alpha \in \fqn : \varphi(\alpha, \beta) = 0 \text{ for all } \beta\in \fqn \}.\]
If rad($\Phi) = \{0\}$, then $\Phi $ is said to be non-degenerate.
\end{definition}

Let $\mathcal B=\{v_1,\dots, v\}$ be a basis of $\fqn$ over $\fq$. The $n \times n$ matrix $A=(a_{ij})$ defined by 
$$a_{ij}= \begin{cases} \Phi(v_i),& \text{if $i=j$}\\
\frac 12(\Phi(v_i+v_j)-\Phi(v_i)-\Phi(v_j)),& \text{if $i\ne j$}.
\end{cases}
$$
is the {\em associated matrix} of the quadratic form $\Phi$ in the basis $\mathcal B$. In particular, the dimension of  $\text{rad}(\Phi)$  is equal to $n- rank (A).$ 

Let $\Phi: \fq^m \to \fq$ and $\Psi : \fq^n\to \fq$ be quadratic forms where $m \ge n$. Let $U$ and $V$ be associated matrix of $\Phi$ and $\Psi$, respectively. We say that $\Phi$ is equivalent to $\Psi$ if there exists $M \in GL_m(\fq)$ such that 
\[M^T UM =\left(
\begin{array}{c|c}
V & 0 \\ \hline
 0 & 0
\end{array}\right)\in M_m(\fq),\]
where $ GL_m(\fq)$ denotes the set of $m\times m$ invertible matrices over $\fq$ and $M_m(\fq)$ denotes the set of $m\times m $ matrices over $\fq$. 
Furthermore, $\Psi $ is called a {\em reduced form} of $\Phi$ if $\text{rad}(\Psi)= \{0\}. $

The following theorem, which associate quadratic forms and characters sums, will be useful for our results and can be obtained from Theorems $6.26$ and $6.27$ of \cite{Lidl} by a straightforward computation.

\begin{lemma}\label{soma}
Let $U$ be an $n\times n$ non null symmetric matrix over $\fq$ and $l = \rank(U)$. Then there exists $M \in \text{GL}_n(\fq)$ such that $V = MUM^T$ is a diagonal matrix, i.e. $V = \text{diag}(a_1,a_2, \dots, a_l, 0,\dots,0)$ where $a_i \in \fq^*$ for all $i=1,\dots,l$. For the quadratic form 
\[F: \fq^n \to \fq , \quad F(X) = XUX^T \quad \quad X = (x_1, \dots, x_n) \in \fq^n, \]
it follows that
$$\sum_{X \in \fq^n} \tilde\psi\big(F(X)\big)= (-1)^{l(s+1)}\tau^{ls}\eta_2(\delta)q^{n-l/2},$$
where $\delta = a_1 \cdots a_l$. 
\end{lemma}

We note that $\eta_2(\delta)$ independs of the system of coordinates  chosen. 

\section{The number of rational points on the curve $\mathcal X_{d,a,b}$
}\label{sec4}
For $a \in \fqn^*, b\in \fqn$
we  compute the number of $\fqn$-rational points on the curve $\mathcal X_{d,a,b}$  by using   well-known properties of character sums. For that, we have the following lemma.

\begin{lemma}\label{lema41}
Let $a \in \fqn^*, b \in \fqn$ and $d$ an integer that divides $q^n-1$. The number of affine rational points on the curve $\mathcal X_{d,a,b}$ over $\fqn$ is 
\[	\mid\mathcal{X}_{d,a,b}(\fqn) \mid=q^n + \frac{1}{q^n}\sum_{\ell=1}^{d-1}  G_n(\chi_d^\ell) \sum_{c \in \F_{q^n}^*}\psi(cb)\overline{\chi_d}^\ell(-c)\sum_{x \in \F_{q^n}}  \psi\big(cax\Tr(x)\big).\] 
\end{lemma}
\begin{proof} 
It follows from Lemma \ref{sumlemm} that 
\begin{equation}\label{eq2}\begin{aligned}
		\mid\mathcal{X}_{d,a,b}(\fqn)\mid& = \frac{1}{q^n} \sum_{c \in \F_{q^n}} \sum_{x,y \in \F_{q^n}}  \psi\big(c(ax\Tr(x)+b-y^d)\big)\\ 
		& = q^n + \frac{1}{q^n} \sum_{c \in \F_{q^n}^*}   \psi(cb)\sum_{x \in \F_{q^n}}  \psi\big(cax\Tr(x)\big) \sum_{y \in \F_{q^n}}  \psi(-c y^d).\\
	\end{aligned}
\end{equation}
Now, let $y^d=z$ and using Lemma \ref{dpower} we obtain
$$\mid\mathcal{X}_{d,a,b}(\fqn) \mid= q^n + \frac{1}{q^n} \sum_{c \in \F_{q^n}^*}   \psi(cb)\sum_{x \in \F_{q^n}}  \psi\big(cax\Tr(x)\big) \sum_{z \in \F_{q^n}}  \psi(-c z)\left[1+\dots+\chi_d^{d-1}(z)\right].$$
Making the change of variable $w=-cz$ and using Lemma \ref{sumlemm} we have that
$$\mid\mathcal{X}_{d,a,b}(\fqn) \mid=q^n + \frac{1}{q^n} \sum_{c \in \F_{q^n}^*}\sum_{\ell=1}^{d-1} \psi(cb)\overline{\chi_d}^\ell(-c)\sum_{x \in \F_{q^n}}  \psi\big(cax\Tr(x)\big) \sum_{w \in \F_{q^n}}  \psi(w)\chi_d^\ell(w).$$
Therefore
\begin{equation}\label{eq3}
	\mid\mathcal{X}_{d,a,b}(\fqn) \mid=q^n + \frac{1}{q^n}\sum_{\ell=1}^{d-1}  G_n(\chi_d^\ell) \sum_{c \in \F_{q^n}^*}\psi(cb)\overline{\chi_d}^\ell(-c)\sum_{x \in \F_{q^n}}  \psi\big(cax\Tr(x)\big).
\end{equation}
\end{proof}
In order to compute the value of the right-hand side sum of Equation~\eqref{eq3}, we will use the fact that $\Tr(cax\Tr(x))$ defines a quadratic form from $\fqn$ into $\fq$. For now on, let  $Q_c(x) $ be a quadratic form of  $\fqn$ over $\F_q$ defined by $Q_c(x) = \Tr(cx\Tr(x))$ and let $B_c(x,y)$ be the bilinear symmetric form associated to $Q_c$.

\begin{proposition}\label{propdim}
For $c\in\fqn^*$, we have that 
\[\dim_{\F_q}\big(\text{rad}(Q_c)\big)= \begin{cases}
n-1 \text{ if } c \in \F_q^*;\\
n-2 \text{ if } c \in \fqn \setminus \F_q.
\end{cases}\]
\end{proposition}
\begin{proof}
The dimension of the radical of the quadratic form $Q_c$ is given by the dimension of the radical of the bilinear form  $B_c(x,y)$, i.e., the dimension of the subspace generated by the elements  $x \in \fqn$ such that $B_c(x,y) =0$ for all $y \in \F_{q^n}$. We observe that
\begin{equation}\label{eq4}
	\begin{aligned}
		B_c(x,y) & =  \Tr(c(x+y)\Tr(x+y))-\Tr(cx\Tr(x))-\Tr(cy\Tr(y)) \\
		& = \Tr(cx \Tr(y) + cy \Tr(x))\\
		& = \Tr(y) \Tr(cx) + \Tr(x)\Tr(cy)\\
		& = \Tr(y (\Tr(cx)+c\Tr(x))).
	\end{aligned}
\end{equation}
Then, $B_c(x,y) =0$ for all $y \in \F_{q^n}$ if and only if $\Tr(cx)+c\Tr(x) = 0$. Therefore, are interested in computing the dimension of 
$$V=\{x\in\fqn:\Tr(cx)+c\Tr(x) = 0\}.$$
We split the proof into two cases:
\begin{itemize}
	\item For $c\in\fq^*$, $\Tr(cx)+c\Tr(x) = c\Tr(x)+c\Tr(x)=2c\Tr(x)$, that implies $$|V|=|\{x\in\fqn:\Tr(x)=0\}|=q^{n-1}$$
	and then $\dim(V)=n-1$.
	\item For $c\in \fqn \setminus \fq,$ if $x \in \fqn$ is such that $\Tr(x)\neq 0$, then $\frac{\Tr(cx)}{\Tr(x)}\in \fq$ and $\frac{\Tr(cx)}{\Tr(x)}\neq -c$. Therefore for any element $x \in V$ we have that $\Tr(x) =0.$ It follows that $\Tr(cx) = - c\Tr(x) =0, $ then $V = V_1 \cap V_2$, where
	$V_1 = \{x \in \F_{q^n} \mid \Tr(x) = 0\}$ and $V_2 = \{x \in \F_{q^n} \mid \Tr(cx) = 0\}$. Since $c \in\F_{q^n} \setminus \F_q$ and $V_1\neq V_2$, it is direct of the fact $\dim(V_1)=\dim(V_2)=n-1$ to verify that $\dim(V)=n-2.$
\end{itemize}
This completes the proof of our assertion.

\end{proof}

\begin{proposition}\label{det} 
 Let $H$ be the symmetric matrix associated to $Q_c$ and let $\delta$ be as defined in Lemma~\ref{soma} for some basis of $\F_{q^n}$ over $\F_q$. Then
	$$\eta_2(\delta) = \begin{cases}
		\eta_2(c), & \text{ if } c \in \fq^*;\\
		\eta_2(-1),
& \text{ if } c \in \fqn \setminus \fq.
	\end{cases}$$
	\end{proposition}

\begin{proof}
Since $\eta_2(\delta)$ does not depend on the basis, we set a basis $\alpha_0, \alpha_1, \dots, \alpha_{n-1}$ of $\fqn $ over $\fq $ such that
\begin{itemize}
    \item $\alpha_0 =n^{-1}$ and $\Tr(\alpha_i)=0$ for all $1\le i\le n-1$ if $\gcd(n,q)=1$;
    \item $\Tr(\alpha_0)=1$, $\alpha_1=1$ and $\Tr(\alpha_0\alpha_i)=\Tr(\alpha_i)=0$ for all $2\le i\le n-1$ in the case when $\gcd(q, n)\ne 1$.
\end{itemize}

For $0 \le i \le n-1$, let $x_i,y_i  \in \fq$ such that  $y  = \sum_{j=0}^{n-1} y_j\alpha_j$ and $x=  \sum_{j=0}^{n-1} x_j\alpha_j$ and let us denote  $X=(x_0,\dots,x_{n-1})$ and $Y=(y_0,\dots,y_{n-1})$. We recall that $Q_c(x)=XHX^T$ and 
\begin{equation}\label{equivalencia}
	\begin{aligned}
		 B_c(X,Y) & = (X+Y)H(X+Y)^T - XHX^T-YHY^T\\
		 & = XHY^T + YHX^T\\
		 & = YH^TX^T + YHX^T\\
		& = Y(2H)X^T.
	\end{aligned}
\end{equation}
Therefore, we can determine $\delta$ from $B_c(x,y)$ by computing the determinant of the reduced matrix associated to $2H$. In order to do this, we observe that $\Tr(x)=x_0$ 
and
$$\Tr(cx)=\begin{cases}
\displaystyle x_0c_0 +  \sum_{1\le i,j \le n-1} x_ic_j\beta_{i,j}
,& \text{if $p\nmid n$}\\
\displaystyle x_0(c_1+c_0\beta_{0,0}
)+c_0x_1+\sum_{2\le i,j \le n-1} x_ic_j\beta_{i,j}
,& \text{if $p\mid n$},
\end{cases}
$$
where $\beta_{i,j}=\Tr(\alpha_i\alpha_j)$.  We  obtain  expressions for  $\Tr(y)$ and $\Tr(cy)$ by a similar process.
Therefore, it follows from \eqref{eq4} that
{\small \begin{align*}
B_c(x,y)
&=\Tr(y)\Tr(cx)+\Tr(cy)\Tr(x)\\
&=\begin{cases}
\displaystyle 
 2x_0y_0c_0 + \sum_{1\le i,j \le n-1}c_j\beta_{i,j}
 (y_0x_i+x_0y_i), &\text{if $p\nmid n$};\\
 \displaystyle 2x_0y_0(c_1+c_0\beta_{0,0}
 )+c_0x_1y_0+c_0x_0y_1+\sum_{2\le i,j \le n-1}c_j\beta_{i,j}(y_0x_i+x_0y_i)
 , &\text{if $p\mid n$}.
\end{cases}
\end{align*}
}

We obtain 
$$2H=\small{\begin{pmatrix} 
2c_0 &  \displaystyle \sum_{j=1}^{ n-1}\beta_{1,j}
c_j& \cdots & \displaystyle \sum_{j=1}^{ n-1}\beta_{n-1,j}
c_j\\
 \displaystyle \sum_{j=1}^{ n-1}\beta_{1,j}
 c_j& 0& \cdots & 0\\
\vdots & \vdots & \ddots & \vdots \\
\displaystyle \sum_{j=1}^{ n-1}\beta_{n-1,j}
c_j& 0 & \cdots & 0
\end{pmatrix}}$$
in the case when $p\nmid n$ and
$$2H=\small{\begin{pmatrix} 
2(c_1+c_0\beta_{0,0})
& c_0& \displaystyle \sum_{j=2}^{ n-1}\beta_{2,j}
c_j
& \cdots & \displaystyle \sum_{j=2}^{ n-1}\beta_{n-1,j}
c_j\\
c_0
& 0&0& \cdots & 0\\
\displaystyle \sum_{j=2}^{ n-1}\beta_{2,j}
c_j
& 0&0& \cdots & 0\\
\vdots & \vdots & \vdots & \ddots & \vdots \\
\displaystyle \sum_{j=1}^{ n-1}\beta_{n-1,j}
c_j& 0& 0 & \cdots & 0
\end{pmatrix}}$$
in the case when $p\mid n$.

In order to compute the reduced form of $2H$, we observe that
\begin{itemize}
	\item In the case when $c \in \F_q$ then $c_0=c$ and $c_1=c_2 = \cdots = c_{n-1} = 0$ if $p\nmid n$  and   $c_1=c$ and $c_0=c_2 = \cdots = c_{n-1} = 0$ if $p\mid n$.  Therefore the associated reduced matrix is $(2c)$.
	\item In the case when $c \in \fqn \setminus \F_q$, then Proposition~\ref{propdim} implies that either there exists $k\in\{1,\dots,n-1\}$ such that $\sum_{j=1}^{ n-1}\beta_{k,j}
	c_j\neq 0$ if $p\nmid n$ and $c_0\ne 0$ or there exists $k\in\{2,\dots,n-1\}$ such that $\sum_{j=2}^{ n-1}\beta_{k,j}
	c_j\neq 0$ if $p\mid n$. Then straightforward manipulations of the lines and columns shows that $2H$ reduces to a matrix of the form 
	$$\begin{pmatrix} 
		u&v\\
		v&0
	\end{pmatrix},$$
	where $v\ne 0$.
\end{itemize}
 In sum, we have that the quadratic character of the determinant $\delta$ of the reduced matrix of $H$ is given by
 $$\eta_2(\delta) = \begin{cases}
 	\eta_2(c), & \text{ if } c \in \fq^*;\\
 	\eta_2(-1)
 	, & \text{ if } c \in \fqn\setminus\fq.
 \end{cases}$$
This completes the proof.
\end{proof}

Combining Lemma~\ref{soma} and Proposition \ref{propdim} and \ref{det}, we have the following result. 
\begin{theorem}\label{dd} For $c\in\fqn^*$, we have
$$\sum_{x \in \fqn} \psi(cx\Tr(x)) = \begin{cases}
(-1)^{s+1}\eta_2(c)\tau^s q^{n-1/2}, &\text{ if } c \in \F_q^*;\\
q^{n-1}, &\text{ if }  c \in \fqn\setminus\F_q.\\
\end{cases}$$

\end{theorem}
Theorem~\ref{dd} allows us to express the number of rational points in terms of Gauss sums.

\begin{proposition}\label{termsofGauss}
Let $v, D$ and $B$ be as in Definition~\ref{vD}. Then
	$$\mid\mathcal{X}_{d,a,b}(\fqn)\mid=q^n+\sum_{\ell=1}^{d-1} \chi_d^\ell(b) q^{n-1}+N,$$
	where $N=N_{d,a,b,n}$ is given below.
\begin{enumerate} 
	\item If $D$ is odd and $B=0$, then 
		$$N=-\frac{q-1}q\sum_{j=1}^{v-1}  G_n(\chi_d^{jD})\chi_d^{jD}(-a).$$
	\item If $D$ is odd and $B\neq0$, then
		{\small$$N=\frac{1}q\sum_{\ell=1}^{d-1}  G_n(\chi_d^\ell)\Big[(-1)^{s+1}\chi_d^\ell(-aB)\chi_2(B)\tau^s\sqrt{q} G_1\big(\eta_{2D}^{D-2\ell}\big)-\chi_d^\ell(-aB)G_1\big(\eta_{2D}^{-2\ell}\big)\Big].$$}
	\item If $D$ is even and $B=0$, then
		{\small$$N=\frac{q-1}q\left((-1)^{s+1}\tau^s\sqrt{q}\sum_{j=0}^{v-1}G_n(\chi_d^{jD+\frac{D}{2}})\chi_d^{jD+\frac{D}{2}}(-a)-\sum_{j=1}^{v-1}  G_n(\chi_d^{jD})\chi_d^{jD}(-a)\right).$$}
	\item If $D$ is even and $B\neq0$, then
		{\small$$N= \frac{1}q\sum_{\ell=1}^{d-1}  G_n(\chi_d^\ell)\Big[(-1)^{s+1}\chi_d^\ell(-aB)\chi_2(B)\tau^s\sqrt{q} G_1\big(\eta_{D}^{\frac{D}{2}-\ell}\big)-\chi_d^\ell(-aB)G_1\big(\eta_{D}^{-\ell}\big)\Big].$$}
\end{enumerate}	
\end{proposition} 

\begin{proof}

By Lemma~\ref{lema41} and Theorem~\ref{dd}, we have that $|\mathcal{X}_{d,a,b}(\fqn)|$ is equal to
{\small
\begin{equation*}
    q^n+ \frac{1}{q}\sum_{\ell=1}^{d-1}  G_n(\chi_d^\ell)\left[ \sum_{ca \in \F_{q^n}\setminus \fq}\psi(cb)\overline{\chi_d}^\ell(-c)+\sum_{ca \in \F_{q}^*}\psi(cb)\overline{\chi_d}^\ell(-c)\big((-1)^{s+1}\eta_2(ac)\tau^s\sqrt{q}\big)\right].
\end{equation*}}

That can be reewritten as

{\small \begin{equation}\label{eq6}
     q^n+ \frac{1}{q}\sum_{\ell=1}^{d-1}  G_n(\chi_d^\ell)\left[ \sum_{ca \in \F_{q^n}^*}\psi(cb)\overline{\chi_d}^\ell(-c)+\sum_{ca \in \F_{q}^*}\psi(cb)\overline{\chi_d}^\ell(-c)\big((-1)^{s+1}\eta_2(ac)\tau^s\sqrt{q}-1\big)\right].
\end{equation}}

Let $T= \displaystyle\sum_{\ell=1}^{d-1}  G_n(\chi_d^\ell)\sum_{ca \in \F_{q^n}^*}\psi(cb)\overline{\chi_d}^\ell(-c)$. By Lemma~\ref{sumlemm}, if $b=0$, then $T=0$. Otherwise, Lemma~\ref{gausslemm} entails that
\begin{equation}\label{eq5}T=\sum_{\ell=1}^{d-1}\chi_d^\ell(-b)G_n(\chi_d^\ell)G_n(\overline{\chi_d}^\ell)=\sum_{\ell=1}^{d-1} \chi_d^\ell(b) q^{n}.
\end{equation}
Now, we compute $S_\ell=\displaystyle\sum_{z \in \F_{q}^*}\psi\big(z\tfrac{b}{a}\big)\overline{\chi_d}^\ell\big(\tfrac{-z}{a}\big)\big((-1)^{s+1}\eta_2(z)\tau^s\sqrt{q}-1\big)$. 

Assume that $D$ is odd. We recall that $\eta_{D}$ is the restriction of $\chi_d$ to $\fq^*$ and that $\eta_{2D}$ is such that $\eta_{2D}^{2}=\eta_D$. Using this notation, 
$$\begin{aligned}
	S_{\ell}&=(-1)^{s+1}\chi_d^\ell(-a)\tau^s\sqrt{q}\sum_{z \in \F_{q}^*}\psi\big(z\tfrac{b}{a}\big)\eta_{2D}^{D-2\ell}(z)-\chi_d^\ell(-a)\sum_{z \in \F_{q}^*}\psi\big(z\tfrac{b}{a}\big)\eta_{2D}^{-2\ell}(z)\\
	&=(-1)^{s+1}\chi_d^\ell(-a)\tau^s\sqrt{q}\sum_{z \in \F_{q}^*}\tilde\psi(zB)\eta_{2D}^{D-2\ell}(z)-\chi_d^\ell(-a)\sum_{z \in \F_{q}^*}\tilde\psi(zB)\eta_{2D}^{-2\ell}(z).\\
\end{aligned}$$
 We split the proof into two cases.
\begin{itemize}
	\item If $B=0$, then it follows from Lemma~\ref{sumlemm} that
	$$S_{\ell}=\begin{cases}
		0,&\text{ if }D\nmid\ell;\\
	-\chi_d^\ell(-a)(q-1),&\text{ if }D\mid\ell.\\
	\end{cases}$$
	\item If $B\neq0$, then
	$$\begin{aligned}
		S_{\ell}&=(-1)^{s+1}\chi_d^\ell(-a)\tau^s\sqrt{q}\eta_{2D}^{D-2\ell}(B^{-1}) G_1\big(\eta_{2D}^{D-2\ell}\big)-\chi_d^\ell(-a)\eta_{2D}^{-2\ell}(B^{-1})G_1\big(\eta_{2D}^{-2\ell}\big)\\
		&=(-1)^{s+1}\chi_d^\ell(-aB)\chi_2(B)\tau^s\sqrt{q} G_1\big(\eta_{2D}^{D-2\ell}\big)-\chi_d^\ell(-aB)G_1\big(\eta_{2D}^{-2\ell}\big).\\
	\end{aligned}$$
\end{itemize}
Our statement follows from the values of $S_{\ell}$ found and  Equations~\eqref{eq6} and \eqref{eq5}. The case where $D$ is even is obtained similarly.

\end{proof}

\section{Bounds and explict formulas for the number of rational points on $\mathcal X_{d,a,b}$
}\label{sec5}

\subsection{Proof of Theorem~\ref{theoremdcota}} The result follows by a direct employment of Proposition~\ref{termsofGauss} and Lemma~\ref{gausslemm}. 
Let us consider the case when $D$ is odd.
\begin{itemize}
\item If $B=0$, by Proposition \ref{termsofGauss} we have
\begin{align*}
|N_{d,a,b,n}| & = \frac{1}q\left|-(q-1) \sum_{j=1}^{v-1} G_n(\chi_d^{jD})\chi_d^{jD}(-a) \right|\\
& \le \frac{(q-1)}q \sum_{j=1}^{v-1}| G_n(\chi_d^{jD})|\\
& = (q-1)q^{n/2-1}\left(\frac{d}D-1\right).
\end{align*}

\item If $B \neq0$, then 
{\small \begin{align*}
|N_{d,a,b,n}| & = \frac{1}q\left| \sum_{\ell=1}^{d-1} G_n(\chi_d^{\ell})\big[(-1)^{s+1}\chi_d^{\ell}(-aB)\chi_2(B)\tau^s\sqrt{q}G_1(\eta_{2D}^{D-2\ell})-\chi_d^{\ell}(-aB)G_1(\eta_{2D}^{-2\ell})\big]\right|\\
& \le  \frac{1}q\sum_{\ell=1}^{d-1} |G_n(\chi_d^{\ell})|\big[\sqrt{q}|G_1(\eta_{2D}^{D-2{\ell}})| + |G_1(\eta_{2D}^{-2{\ell}})|\big] \\
& \le (d-1)(q^{n/2}+q^{(n-1)/2}).
\end{align*}}

\end{itemize}
Those inequalities along with Proposition~\ref{termsofGauss} assures us the result in the cases when $D$ is odd.  The case where $D$ is even is obtained similarly. $\hfill\qed$

\subsection{Proof of Theorem~\ref{theoremdcase2}}
We have that $$D= \frac{2}{v } = \begin{cases} 1,& \text{ if } n \text{ is even;} \\ 2, &  \text{ if } n \text{ is odd.} \end{cases}$$ To compute $\mid\mathcal{X}_{2,a,b}\mid $ we need to  determine 
$N$ in  Proposition~\ref{termsofGauss}. 
In the case when $n$ is even we have $D=1$, therefore 
\begin{enumerate}
\item If $B=0$, then 
$$\begin{aligned}
N & = \frac{1}q \left(-(q-1)G_n(\chi_2)\chi_2(-a)\right)\\
& = q^{\frac{n-2}2}(q-1)(-1)^{ns+2}\tau^{ns}\chi_2(-a)\\
& = q^{\frac{n-2}2}(q-1)\tau^{ns}\chi_2(-a).
\end{aligned}$$
\item If $B \neq 0$:
		$$\begin{aligned} N& =\frac{1}q\left(G_n(\chi_2)\Big[(-1)^{s+1}\chi_2(-aB)\chi_2(B)\tau^s\sqrt{q} G_1\big(\eta_{2}\big)-\chi_2(-aB)G_1\big(\eta_{2}^{2}\big)\Big]\right)\\
		& =  (-1)^{ns+1}\tau^{ns}q^{\frac{n-2}2}\Big[(-1)^{s+1}\chi_2(-a)\tau^{2s}(-1)^{s+1} q+\chi_2(-aB)\Big]\\
& =- \tau^{ns}q^{\frac{n-2}2}\Big[\chi_2(-a)\tau^{2s}q+\chi_2(-aB)\Big].\\  \end{aligned}$$
\end{enumerate}
In the case when $n$ odd, we have that $D$ is even,  then  \begin{enumerate}
\item If $B=0$:
$$\begin{aligned} N&=\frac{q-1}q\left((-1)^{s+1}\tau^s\sqrt{q}G_n(\chi_2)\chi_2(-a)\right)\\
& =\frac{q-1}{q}(-1)^{s+1}\tau^s(-1)^{ns-1}\tau^{ns}q^{\frac{n+1}2}\chi_2(-a)\\
& =\tau^{(n+1)s}(q-1)\chi_2(-a)q^{\frac{n-1}2}. \\
\end{aligned}$$

\item If $B \neq 0$:
			$$\begin{aligned} N& = \frac{1}q G_n(\chi_2)\Big[(-1)^{s+1}\chi_2(-aB)\chi_2(B)\tau^s\sqrt{q} G_1\big(\eta_{2}^{0}\big)-\chi_2(-aB)G_1\big(\eta_{2}\big)\Big]\\
							& = (-1)^{ns-1}\tau^{ns}q^{\frac{n-2}2}\Big[(-1)^{s+1}\chi_2(-a)\tau^s\sqrt{q} (-1)-\chi_2(-aB)(-1)^{s-1}\tau^s \sqrt{q}\Big]\\
							& =-\tau^{s(n+1)}q^{\frac{n-1}2}\Big[\chi_2(-a) +\chi_2(-aB)\Big].\\
\end{aligned}$$
 Combined with Proposition \ref{termsofGauss}, the theorem follows. $\hfill\qed$

\end{enumerate}

\subsection{Proof of Theorem~\ref{teores}}
Let $M$ be the number of elements $\alpha\in\fqn$ such that both $\alpha$ and $\Tr(\alpha)$ are quadratic residues in $\fqn$. Let $\Lambda_2$ be the set of quadratic residues in $\fqn^*$ and $\Lambda=\fqn^*\backslash\Lambda_2$. We recall that, for $x\in\fqn$,  $\chi_2(x)=1$ if and only if $x\in\Lambda_2$. Then we have that
$$\big[1+\chi_2(x)\big]\big[1+\chi_2(\Tr(x))\big]=\begin{cases}
    0,&\text{ if either }x\in\Lambda\text{ or }\Tr(x)\in\Lambda;\\
    1,&\text{ if }x=0;\\
    2,&\text{ if }x\in\Lambda_2\text{ and }\Tr(x)=0;\\
    4,&\text{ if }x\in\Lambda_2\text{ and }\Tr(x)\in\Lambda_2.\\
\end{cases}$$
Therefore, Schur's orthogonality relations (Lemmas~\ref{sumlemm} and \ref{dpower}) imply that
\begin{equation}\label{equati5}
  M =\tfrac{1}{2}+\tfrac{1}{4}\sum_{x\in\fqn}\big[1+\chi_2(x)\big]\big[1+\chi_2(\Tr(x))\big]+\tfrac{1}{4q}\sum_{x\in\fqn}\big[1+\chi_2(x)\big]\sum_{c\in\fq}\tilde\psi(c \Tr(x)).
  \end{equation}
We note that Lemma~\ref{item433} implies that
\begin{equation}\label{equati1}
\sum_{x\in\fqn}\big[1+\chi_2(x)\big]\big[1+\chi_2(\Tr(x))\big]=\sum_{x\in\fqn}\big[1+\chi_2(\Tr(x))+\chi_2(x\Tr(x))\big].
\end{equation}
Since $n$ is odd, $\chi_2$ is also the multiplicative character of order 2 over $\fq$. Since $\Tr(x)$ is a linear transformation over $\fqn$ whose image is $\fq$, each element in $\fq$ has $q^{n-1}$ elements in its preimage. Therefore, 
\begin{equation}\label{equati2}
    \sum_{x\in\fqn}\chi_2(\Tr(x))=q^{n-1}\sum_{y \in \fq}\chi_2(y) = 0.
\end{equation}
We note that
\begin{equation}\label{equati3}
    |\mathcal{X}_{2,1,0}(\fqn)|=\sum_{x\in\fqn}\big[1+\chi_2(x\Tr(x))\big]=q^n+\sum_{x\in\fqn} \chi_2(x\Tr(x)).
\end{equation}
From \eqref{equati1}, \eqref{equati2} and \eqref{equati3}, it follows that
\begin{equation}\label{equati4} \sum_{x\in\fqn}\big[1+\chi_2(x)\big]\big[1+\chi_2(\Tr(x))\big]=|\mathcal{X}_{2,1,0}(\fqn)|.
\end{equation}
For the last sum in \eqref{equati5}, Lemma~\ref{sumlemm} implies that
\begin{equation}\label{equati7}
\sum_{x\in\fqn}\big[1+\chi_2(x)\big]\sum_{c\in\fq}\tilde\psi(c \Tr(x))=q^{n-1}\cdot q+\sum_{c\in\fq}\sum_{x\in\fqn}\chi_2(x)\tilde\psi(c \Tr(x)).
\end{equation}
The map $x\mapsto\Tr(x))$ is a linear map over $\fq$ so that $c\Tr(x)=\Tr(cx)$  for all $c\in\fq$ and then $\tilde\psi(c\Tr(x)=\tilde\psi(\Tr(cx))=\psi(cx)$. Therefore, by setting $z=cx$, we obtain that
\begin{equation}\label{equati6}
\begin{aligned}
\sum_{c\in\fq}\sum_{x\in\fqn}\chi_2(x)\tilde\psi(c \Tr(x))&=\sum_{c\in\fq}\chi_2(c)\sum_{z\in\fqn}\chi_2(z)\psi(z)\\
&=G_n(\chi_2)\sum_{c\in\fq}\chi_2(c)\\
&=0,\\
\end{aligned}
\end{equation}
where the last equality follows by Lemma~\ref{item433}. In sum, Equations~\eqref{equati5}, \eqref{equati4}, \eqref{equati7} and \eqref{equati6} imply that
$$M=\frac{|\mathcal{X}_{2,1,0}(\fqn)|+q^{n-1}+2}{4}.$$
Now our result follows by Theorem~\ref{theoremdcase2}.$\hfill\qed$

\subsection{Proof of Theorem~\ref{casebeuquals0}} \begin{enumerate}
\item Assume that $v$ is $(p,r)$-admissible and $D$ is odd. Proposition~\ref{termsofGauss} states that
\begin{align*}
    N=-\frac{q-1}q\sum_{j=1}^{v-1}  G_n(\chi_d^{jD})\chi_d^{jD}(-a).
\end{align*}
We recall that $\chi_d^D$ has order $\tfrac{d}{D}=v$. If $v=1$, then $N=0$. If $v=2$, then Theorem~\ref{gaussdois} entails that
\begin{align*}
    N& =-(q-1)(-\tau^{sn} q^{\frac{n-2}{2}}\chi_2(-1)\chi_2(a))\\ &=(q-1)\tau^{sn} q^{\frac{n-2}{2}}\chi_2(a),
\end{align*}
since $\chi_2(-1)=\tau^{2ns}=1$. If $v>2$, then Theorem~\ref{item405} implies that
$$\begin{aligned}N&=-\frac{q-1}q(-\varepsilon) q^{\frac{n}2}\sum_{j=1}^{v-1}\big(\varepsilon^{uD}\chi_d^{D}(-a)\big)^j\\
	&=q^{\frac{n-2}2}(q-1)\varepsilon\theta_{v}(-a,\varepsilon^{uD}).\\
\end{aligned}$$

\item Let us assume that $2v$ is $(p,r)$-admissible and $D$ is even. By Proposition~\ref{termsofGauss} we have 
\[N =\frac{q-1}q\left((-1)^{s+1} \tau^s \sqrt{q} \sum_{j=0}^{v-1}G_n(\chi_d^{jD+\frac{D}2})\chi_d^{jD+\frac{D}2}(-a) -
\sum_{j=1}^{v-1}G_n(\chi_d^{jD})\chi_d^{jD}(-a)\right).\] 
If $v=1$, $d=D$ and it follows that 
\begin{align*}
N &=\frac{q-1}q\left( (-1)^{s+1}\tau^s \sqrt{q}G_n(\chi_d^ {\frac{d}2})\chi_d^{\frac{d}2}(-a)\right)\\
&  = \frac{q-1}{q}(-1)^{s+1}\tau^s\sqrt{q}G_n(\chi_2)\chi_2(-a)\\
& =(-1)^{ns+s}(q-1)q^{\frac{n-1}2} \tau^{s(n+1)} \chi_2(a)\\ 
& =(-1)^{s}(q-1)q^{\frac{n-1}2} \tau^{s(n+1)} \chi_2(a).
\end{align*}
If $v = 2, $ we have $d =2D$. By hypotheses, $4$ is $(p,r)$-admissible, and $u = \frac{p^r+1}{d}= \frac{p^r+1}{4 \cdot D/2}$. Then 
{\small \begin{align*}
N &= \frac{q-1}{q} \left((-1)^{s+1}\tau^s \sqrt{q}(G_n(\chi_d^{\frac{d}4})\chi_d^{\frac{d}4}(-a)+ G_n(\chi_d^{\frac{3d}4})\chi_d^{\frac{3d}4}(-a))-G_n(\chi_d^{\frac{d}2})\chi_d^{\frac{d}2}(-a)\right)\\
&  =\frac{q-1}{q} \left( (-1)^{s+1}\tau^s \sqrt{q}(G_n(\chi_4)\chi_4(-a)+G_n(\overline{\chi_4})\overline{\chi_4}(-a))-G_n(\chi_2)\chi_2(-a)\right)\\
& =\frac{q-1}{q} \left( (-1)^{s+1}\tau^s \sqrt{q}(-\varepsilon^{u\frac{D}2+1}q^{\frac{n}2}\chi_4(-a)+\chi_4(-1)(-\varepsilon^{u\frac{D}2+1}q^{\frac{n}2})\overline{\chi_4}(-a))+(-1)^{sn}\tau^{ns}q^{\frac{n}2}\chi_2(-a)\right)\\
& =  q^{\frac{n-2}2}(q-1)\left((-1)^{s}\tau^s \sqrt{q}\varepsilon^{u\frac{D}2+1}(\chi_4(-a)+\overline{\chi_4}(a))+\tau^{ns}\chi_2(a)\right).\\
\end{align*}}
 If $v>2$, then Theorem~\ref{item405} implies that
 {\small \begin{align*}
 N & =\frac{q-1}q\left( (-1)^{s+1}\tau^sq^{\frac{n+1}2}\chi_{2v}(-a)(-\varepsilon^{1+\frac{uD}2})\sum_{j=0}^{v-1} \big(\varepsilon^{2u}\chi_{v}(-a)\big)^j - q^{\frac{n}2}(-\varepsilon) \sum_{j=1}^{v-1}\big(\varepsilon^{uD}\chi_{v}(-a)\big)^j\right)\\
& =\frac{q-1}q\left((-1)^{s}\tau^sq^{\frac{n+1}2}\chi_{2v}(-a)\varepsilon^{1+\frac{uD}2}[1+\theta_{v}(-a,1)]+q^{\frac{n}2}\varepsilon \theta_{v}(-a,1)\right)\\
&= (q-1)q^{\frac{n-2}2}\varepsilon\Big[(-1)^{s}\tau^s\sqrt{q}\chi_{2v}(-a)\varepsilon^{\frac{uD}2}[1+\theta_{v}(-a,1)]+ \theta_{v}(-a,1)\Big].
 \end{align*}}
\end{enumerate} $\hfill\qed$

\section{The case $\Tr(b/a)\neq 0$}\label{sec7}

\begin{definition}
    If $D$ is $(p,r_0)$-admissible, we define
$$ \varepsilon_0=(-1)^{\frac{s}{2r_0}},\quad   \text{ and } \quad u_0=\begin{cases}
\tfrac{p^{r_0}+1}{2D},&\text{if $D$ is odd;}\vspace{2mm}\\ 
\tfrac{p^{r_0}+1}{D},&\text{if $D$ is even.}
\end{cases}$$
\end{definition}

\begin{theorem}\label{theoremdcase}
Let $a,b \in\fqn$ such that $B\neq0$. 
If $d$ is $(p,r)$-admissible, $D$ is odd and  $(p,r_0)$-admissible, 
then 
	$$N=N_1 q^{\frac{n}{2}}+N_2 q^{\frac{n-1}{2}}+N_3 q^{\frac{n-2}{2}},$$
	 where
\begin{itemize}
	\item $N_1=(-1)^{s+1}\varepsilon\varepsilon_0^{1+u_0D}\tau^s\chi_2(B) \theta_d(-aB,\varepsilon^u);$
	\item $N_2=\varepsilon\varepsilon_0\left(-\theta_d(-aB,\varepsilon^u)+\theta_{v}(-aB,\varepsilon^{u})\right);$
	\item $N_3=-\varepsilon \theta_{v}(-aB,\varepsilon^{u});$
\end{itemize}
\end{theorem}

\begin{proof}
By hypotheses $d$ is $(p,r)$-admissible and $D$ is odd. By Proposition~\ref{termsofGauss},
$$N=\frac{1}q\sum_{\ell=1}^{d-1}  G_n(\chi_d^\ell)\Big[(-1)^{s+1}\chi_d^\ell(-aB)\chi_2(B)\tau^s\sqrt{q} G_1\big(\eta_{2D}^{D-2\ell}\big)-\chi_d^\ell(-aB)G_1\big(\eta_{2D}^{-2\ell}\big)\Big].$$
 Since $D$ is $(p,r_0)$-admissible and $p^{r_0}+1$ is even, it follows that $2D$ divides $p^{r_0}+1$. Therefore, Theorem~\ref{item405} and Corollary~\ref{corolchartwo} imply that
$$\begin{aligned}
	N&=-q^{\frac{n-2}{2}} \sum_{\ell=1}^{d-1} \varepsilon^{u\ell+1}\chi_d^\ell(-aB)\Big[(-1)^s\tau^s\chi_2(B) \varepsilon_0^{u_0(2\ell+D)+1}q+\varepsilon_0^{2u_0\ell+1} \sqrt{q}\Big]+R\\
	&=-q^{\frac{n-1}{2}} \varepsilon\varepsilon_0\Big((-1)^s\tau^s\chi_2(B) \varepsilon_0^{u_0D}\sqrt{q}+1\Big)\sum_{\ell=1}^{d-1} \big(\chi_d(-aB)\varepsilon^u\varepsilon_0^{2u_0}\big)^{\ell}+R\\
	&=-q^{\frac{n-1}{2}} \varepsilon\varepsilon_0\Big((-1)^s\tau^s\chi_2(B) \varepsilon_0^{u_0D}\sqrt{q}+1\Big)\theta_d(-aB,\varepsilon^u)+R,\\
\end{aligned}$$
where $R$ is the sum of the terms in the case when the multiplicative caracter is trivial, i.e.,
$$\begin{aligned}
	R&=\frac{1}{q} \left(q^{\frac{n}{2}} \sum_{j=1}^{d/D-1} \varepsilon^{ujD+1}\chi_d^{jD}(-aB)\Big[-1+\sqrt{q}\varepsilon_0^{2u_0jD+1}\Big]\right)\\
	&=-q^{\frac{n-2}{2}} \varepsilon\sum_{j=1}^{d/D-1}\Big[ \big(\varepsilon^{uD}\chi_d^{D}(-aB)\big)^j-\varepsilon_0\big(\varepsilon^{uD}\varepsilon_0^{2u_0D}\chi_d^{D}(-aB)\big)^j\sqrt{q}\Big]\\
	&=q^{\frac{n-2}{2}} \varepsilon(-1+\varepsilon_0\sqrt q)\theta_{v}(-aB,\varepsilon^{u}).\\
\end{aligned}$$
Rearranging the terms, we obtain the expression
$$N=N_1 q^{\frac{n}{2}}+N_2 q^{\frac{n-1}{2}}+N_3 q^{\frac{n-2}{2}},$$
proving the statement.
\end{proof}

  \begin{theorem}\label{theoremdcaseeven}
	Assume that $B\neq0$ and $d>2$. 	We denote $\alpha_1 =\theta_d(-aB,\varepsilon^u\varepsilon_0^{u_0})$ and $  \alpha_2 =\theta_{v}(-aB,1) .$
	If $d$ is $(p,r)$-admissible,  $D$ is even and  $(p,r_0)$-admissible, 
	then 
	$$N=N_1 q^{\frac{n}{2}}+N_2 q^{\frac{n-1}{2}}+N_3 q^{\frac{n-2}{2}},$$
	where
	{\small \begin{itemize}
	\item $N_1=(-1)^s\varepsilon\varepsilon_0\tau^s\chi_2(B)\big[-\varepsilon_0^{u_0\frac{D}2}\alpha_1+\varepsilon^{u\frac{D}2}\chi_{2v}(-aB)(1+\alpha_2)\big],$ 
	 \item  $ N_2=\varepsilon\varepsilon_0\big(\alpha_2-\alpha_1\big)-(-1)^{s}\varepsilon^{u\frac{D}2+1}\tau^s\chi_{2v}(-aB)\chi_2(B)(1+\alpha_2),$
	\item $N_3=-\varepsilon \alpha_2,$
\end{itemize}}
\end{theorem}
\begin{proof}
By hypotheses $d$ is $(p,r)$-admissible and $D$ is even. By Proposition~\ref{termsofGauss},
 $$N = \frac{1}q\sum_{\ell=1}^{d-1}G_n(\chi_d^\ell)\Big[(-1)^{s+1}\chi_d^\ell(-aB)\chi_2(B) \tau^s \sqrt{q} G_1(\eta_D^{\frac{D}2-\ell})-\chi_d^\ell(-aB)G_1(\eta_D^{-\ell})\Big].$$
Therefore, since $D$ is $(p,r_0)$ admissible
$$\begin{aligned}
	N&=-q^{\frac{n-2}{2}} \sum_{\ell=1}^{d-1} \varepsilon^{u\ell+1}\chi_d^\ell(-aB)\Big[-(-1)^{s+1}\tau^s\chi_2(B) \varepsilon_0^{u_0(\ell+\frac{D}2)+1}q+\varepsilon_0^{u_0\ell+1}\sqrt{q}\Big]+R\\
	&=-q^{\frac{n-1}{2}} \varepsilon\varepsilon_0\Big((-1) ^s\tau^s\chi_2(B) \varepsilon_0^{u_0\frac{D}2}\sqrt{q}+1\Big)\sum_{\ell=1}^{d-1} \big(\chi_d(-aB)\varepsilon^u\varepsilon_0^{u_0}\big)^{\ell}+R,\\
	&=-q^{\frac{n-1}{2}} \varepsilon\varepsilon_0\Big((-1)^s
	\tau^s\chi_2(B) \varepsilon_0^{u_0\frac{D}2}\sqrt{q}+1\Big)\theta_d(-aB,\varepsilon^u\varepsilon_0^{u_0})+R,\\
\end{aligned}$$

where 
$R=R_1+R_2$, 
{\small $$\ \hspace{-1cm}\begin{aligned}
R_1 & =- q^{\tfrac{n-2}2} \sum_{j=0}^{d/D-1} \varepsilon^{u\big(\tfrac{D}2(2j+1)\big)+1}\chi_d^{\tfrac{D}2(2j+1)}(-aB)(-1)^{s+1}\big[-\chi_2(B)\sqrt{q}\tau^s+\tau^s\chi_2(B)\varepsilon_0^{u_0\big(\frac{D}2(2j+1)+\frac{D}2\big)+1}q\big]\\
& = (-1)^{s}q^{\tfrac{n-1}2} \varepsilon^{u\frac{D}2+1}\tau^s\chi_{2v}(-aB)\chi_2(B)\sum_{j=0}^{d/D-1}\big(\varepsilon^{uD}\chi_{v}(-aB)\big)^j(-1+\sqrt{
q} \varepsilon_0^{u_0Dj+u_0D+1})\\
& =(-1)^{s} q^{\tfrac{n-1}2} \varepsilon^{u\frac{D}2+1}\tau^s\chi_{2v}(-aB)\chi_2(B)\sum_{j=0}^{d/D-1}\big[-\big(\varepsilon^{uD}\chi_{v}(-aB)\big)^j+\big(\varepsilon^{uD}\varepsilon_0^{u_0D}\chi_{v}(-aB)\big)^j\varepsilon_0^{u_0D+1}\sqrt{
q}\big] \\
& =(-1)^{s} q^{\tfrac{n-1}2} \varepsilon^{u\frac{D}2+1}\tau^s\chi_{2v}(-aB)\chi_2(B)\big[-(1+\theta_{v}(-aB, \varepsilon^{uD}))+(1+\theta_{v}(-aB, \varepsilon^{uD}\varepsilon_{0}^{u_0D}))\varepsilon_0^{u_0D+1}\sqrt{q}\big] \\
& = (-1)^{s} q^{\tfrac{n-1}2} \varepsilon^{u\frac{D}2+1}\tau^s\chi_{2v}(-aB)\chi_2(B)(1+\theta_v(-aB,1))(\varepsilon_0^{\frac{uD}2}\sqrt{q}-1)
\end{aligned}$$}
and
$$\begin{aligned}
	R_2&=-q^{\frac{n-2}{2}} \sum_{j=1}^{d/D-1} \varepsilon^{ujD+1}\chi_d^{jD}(-aB)\Big[1-\sqrt{q}\varepsilon_0^{u_0jD+1}\Big]\\
	&=-q^{\frac{n-2}{2}} \varepsilon\sum_{j=1}^{d/D-1}\Big[\big(\varepsilon^{uD}\chi_{v}(-aB)\big)^j-\varepsilon_0\big(\varepsilon^{uD}\varepsilon_0^{u_0D}\chi_{v}(-aB)\big)^j\sqrt{q}\Big]\\
	&=-q^{\frac{n-2}{2}}\varepsilon \Big[ \theta_{v}(-aB,1)-\varepsilon_0\theta_{v}(-aB,1)\sqrt{q}\Big]\\
	& = -q^{\frac{n-2}2}\varepsilon\theta_v(-aB,1)(1-\varepsilon_0\sqrt{q}).
\end{aligned}$$
Altogether, we have shown that
$$N=N_1 q^{\frac{n}{2}}+N_2 q^{\frac{n-1}{2}}+N_3 q^{\frac{n-2}{2}},$$
proving the statement.
\end{proof}


\printbibliography

\end{document}